\documentclass{amsart}
\usepackage{amsmath}
\usepackage{hyperref} 
\usepackage{amsfonts, yfonts, stackrel}
\usepackage{amssymb}
\usepackage{amsthm}
\usepackage[arrow, matrix, curve]{xy} 

\newtheorem{proposition}{\textbf{Proposition}}

\newtheorem{corollary}[proposition]{\textbf{Corollary}}
\newtheorem{theorem}[proposition]{\textbf{Theorem}}

\theoremstyle{definition}
\newtheorem{definition}[proposition]{\textbf{Definition}}

\newtheorem{remark}[proposition]{\textbf{Remark}}

\numberwithin{proposition}{section}

\begin{document}

\title{{Harmonic Maps and Hypersymplectic Geometry}}

\author{Markus R\"oser} 
\address{Mathematisches Institut \\ Westf\"alische Wilhelms-Universit\"at M\"unster \\
Einsteinstrasse 62, M\"unster, Germany} 
\email{mroes\_02@ uni-muenster.de}

\maketitle
\begin{abstract}
We study the hypersymplectic geometry of the moduli space of solutions to Hitchin's harmonic map equations on a $G$-bundle. This is the split-signature analogue of Hitchin's Higgs bundle moduli space. Due to the lack of definiteness, this moduli space is globally not well-behaved. However, we are able to construct a smooth open set consisting of solutions with small Higgs field, on which we can investigate the hypersymplectic geometry. Finally, we reinterpret our results in terms of the Riemannian geometry of the moduli space of $G$-connections.
\end{abstract}

\tableofcontents

Mathematics Subject Classifiaction:  53C25,   53C07, 53D30, 3D20

\section{Introduction}
A \emph{hypersymplectic manifold} is a quintuple $(M^{4k},g,I,S,T)$, where $M$ is a $4k$-dimensional real manifold, $g$ is a pseudo-Riemannian metric of signature $(2k,2k)$ and $I,S,T$ are skew-adjoint sections of $\mathrm{End}(TM)$ such that 
$$S^2 = T^2 = \mathrm{id}_{TM} = - I^2\qquad IS = T = -SI,$$
and $$\nabla^g I = \nabla^g S = \nabla^gT = 0,$$
where $\nabla^g$ is the Levi-Civita connection associated with $g$. 
The skew-adjointness and the covariant constancy of $I,S,T$ imply that 
$$\omega_I = g(I-,-) \qquad \omega_S = g(S-,-) \qquad \omega_T = g(T-,-)$$
define symplectic forms on $M$, hence the name hypersymplectic. The above definition is reminiscent of the definition of a hyperk\"ahler manifold and the existence of the parallel endomorphisms $I,S,T$ implies that the holonomy of $g$ is contained in the non-compact Lie group $\mathrm{Sp}(2k,\mathbb R)$, which is the split-real form of $\mathrm{Sp}(2k,\mathbb C)$. In this way hypersymplectic manifolds are neutral-signature cousins of hyperk\"ahler manifolds, whose holonomy is contained in the compact real form $\mathrm{Sp}(k)$. Due to the common complexification of the holonomy groups, many facts from hyperk\"ahler geometry carry over to hypersymplectic manifolds. In particular, hypersymplectic manifolds are complex symplectic and Ricci-flat. 

In the hyperk\"ahler situation we have a whole two-sphere of compatible K\"ahler structures, the unit sphere in the three-dimensional real vector space spanned by $I,J,K$, which we may identify with the space of imaginary quaternions.
Hypersymplectic structures, however, are in some sense less symmetric, since we do not only deal with in fact a whole two-sheeted hyperboloid of complex structures, but also with a connected hyperboloid of product structures. Again, these hyperboloids can be thought of as spheres in the three-dimensional vector space spanned by $I,S,T$, which carries naturally a metric of Lorentz signature. The hyperboloids mentioned above then correspond to the subsets of spacelike, respectively timelike, unit vectors. So there are always two different ways to look at hypersymplectic manifolds. One can study them from the point of view of complex, or in fact pseudok\"ahler, geometry. Or one can focus on the product structures and investigate the parak\"ahler geometry.  

A powerful tool to construct examples of hypersymplectic manifolds is the hypersymplectic quotient construction, which is an adaptation of the symplectic reduction of Marsden and Weinstein to the hypersymplectic setting. It is closely analogous to the hyperk\"ahler quotient construction, however, the non-trivial signature of the hypersymplectic metric gives rise to pathologies not present in the hyperk\"ahler case. The precise statement is the following. 

\begin{theorem}[\cite{Hitchin:1990a}]
Let $(M,g,I,S,T)$ be a hypersymplectic manifold and let $G$ be a Lie group which acts on $M$ preserving the hypersymplectic structure. Suppose the action is hamiltonian with respect to each of the symplectic structures $\omega_i$ with moment maps $\mu_i: M \to \mathfrak g^*$ for $i\in\{I,S,T\}$. Define the hypersymplectic moment map 
$$\mu = (\mu_I,\mu_S,\mu_T): M \to \mathfrak g \otimes \mathbb R^3,$$
and assume that 
\begin{enumerate}
\item $c \in Z(\mathfrak g^*) \otimes \mathbb R^3$ is a regular value of $\mu$, \label{s}
\item $G$ acts freely and properly on $\mu^{-1}(c)$, \label{f}
\item the metric $g$ restricted to the tangent spaces to the $G$-orbits in $\mu^{-1}(c)$ is non-degenerate. \label{d}
\end{enumerate}
Then the quotient metric on $\mu^{-1}(c)/G$ is again hypersymplectic and the symplectic forms on the quotient $\tilde\omega_I,\tilde\omega_S,\tilde\omega_T$ satisfy $i^*\omega_i = p^*\tilde\omega$ for all $i\in\{I,S,T\}$, where $p:\mu^{-1}(c) \to \mu^{-1}(c)/G$ is the projection and $i: \mu^{-1}(c) \to M$ is the inclusion map.
\end{theorem}
Conditions (\ref{s}) and (\ref{f}) ensure that the quotient is a smooth manifold, whereas the third condition guarantees the non-degeneracy of the symplectic forms induced by the $\omega_i$'s. In the hyperk\"ahler case, all we have to assume is condition (\ref{f}), which then implies (\ref{s}) and (\ref{d}). The definiteness of the hyperk\"ahler metric is crucial in the proof and these arguments do not work anymore in our situation. We still have however that (\ref{f}) and (\ref{d}) together imply (\ref{s}). 

In most applications, we can ensure that the conditions (\ref{s}) and (\ref{f}) are satisfied, e.g. by changing the level $c$ or by passing to a suitable subgroup or  quotient of $G$. Condition (\ref{d}) is more difficult to arrange and therefore the hypersymplectic structure on $\mu^{-1}(c)/G$ is expected to be degenerate even though the quotient manifold itself is smooth, i.e. the $\omega_i$'s are symplectic only on the complement of a \emph{degeneracy locus}. For more details, see \cite{Hitchin:1990a} and \cite{DancerSwann:2008}.

In this paper, we investigate a hypersymplectic  structure obtained from the quotient construction in an infinite-dimensional setting, namely we study the hypersymplectic geometry of the moduli space of solutions to Hitchin's gauge theoretic equations for harmonic maps from a Riemann surface $M$ into a compact Lie group $G$ \cite{Hitchin:1990}. Historically, these equations mark the starting point of the subject of hypersymplectic geometry \cite{Hitchin:1990a} and are given by 
\begin{equation}\tag{4}\label{H}
F^\nabla = [\Phi\wedge \Phi^*],\qquad \bar\partial^\nabla \Phi =0,
\end{equation}
where $(\nabla,\Phi)$ is a pair consisting of a $G$-connection $\nabla$ on a $G$-vector bundle $E$, and a Higgs field $\Phi \in \Gamma(M, \mathfrak g(E)^\mathbb C \otimes K)$, where $K$ is the canonical bundle and $\mathfrak g(E) \subset \mathrm{End}(E)$ is the associated bundle of Lie algebras with fibre $\mathrm{Lie}(G)$. They can be obtained from the ASD equations on $\mathbb R^{2,2}$ by imposing translation invariance with respect to the $s$ and $t$ directions, and they are the split signature analogue of Hitchin's self-duality equations on a Riemann surface \cite{Hitchin:1987}. The moduli space of solutions to the latter equations is known to carry a hyperk\"ahler structure, essentially because the equations may be interpreted as the vanishing condition of a hyperk\"ahler moment map in an infinite-dimensional setting. Since the harmonic map equations and the self-duality equations only differ by a sign in the first equation, it is natural to expect the harmonic map equations to have a moment map interpretation, too. In this way Hitchin in \cite{Hitchin:1990a} was led to the definition of hypersymplectic structures. 

Solutions to the harmonic map equations (\ref{H}) describe harmonic sections  of flat $G\times G$-bundles. Due to the split-signature origin of the equations, we cannot expect to have a smooth global hypersymplectic moduli space in this situation. Typically, the smoothness of gauge-theoretic moduli spaces is established by viewing the moduli space as the vanishing locus of a certain section of a vector bundle (with infinite-dimensional fibres). One then proves a vanishing theorem, which asserts that at each point of the moduli space the differential of the section is surjective, or at least of constant co-rank. It is at this point, where in addition to the ellipticity, the \emph{definiteness} of the involved operators enters the argument. In our situation the change of signature destroys the positivity of the involved operators and we are thus not able to prove vanishing theorems which would ensure that the dimension of the moduli space does not jump. However, if we are looking for solutions with zero Higgs field, the sign change does not play a role and an argument involving the implicit function theorem enables us to produce a well-behaved neighbourhood of the moduli space of flat connections inside the moduli space of solutions to Hitchin's harmonic map equations. On this neighbourhood we can study the hypersymplectic geometry of the moduli space.

Reinterpreting the equations (\ref{H}) as the equation for a geodesic segment on the space $\mathcal A/\mathcal G$ of $G$-connections modulo gauge transformations, whose endpoints lie on the moduli space $\mathcal N$ of flat connections (see \cite{Hitchin:1990}), we exhibit this neighbourhood as a neighbourhood of the diagonal inside $\mathcal N \times \mathcal N$. The local product structure of this neighbourhood is induced by the endomorphism $S$ of the hypersymplectic structure on the moduli space. It corresponds to assigning to a geodesic segment its endpoints. 
Moreover,  geodesic segments with conjugate endpoints can be related to elements of the degeneracy locus for the hypersymplectic structure. In classical terminology of  Riemannian geometry, we show that the degeneracy locus of the hypersymplectic structure is directly related to the \emph{cut locus} of the infinite-dimensional Riemannian manifold $\mathcal A/\mathcal G$. 

\section{The Equations and Their Moduli Space of Solutions}

Although the equations make sense for any compact structure group $G$, we will work in the following with $G=\mathrm U(n)$, i.e. consider the equations on a hermitian vector bundle $E$ of rank $n$ over a compact Riemann surface $M$. Since every compact Lie group may be embedded into $\mathrm U(n)$ for some $n$, this is not a restriction and with minor modifications our proofs  work for arbitrary vector bundles with compact structure group. In this section, we study the analysis of the equations and show that inside the moduli space of solutions, there exists a smooth open set represented by solutions with small Higgs fields. The proof is a deformation argument based on the implicit function theorem. But first we set up the framework and notation.

Let $M$ be a compact Riemann surface of genus $g$ and let $E \to M$ be a hermitian vector bundle. Let $\mathcal A$ be the space of unitary connections on $E$. The equations are given by
$$F^\nabla = [\Phi\wedge \Phi^*],\qquad \bar\partial^\nabla \Phi =0,$$
where $(\nabla,\Phi)$ is a pair consisting of a unitary connection $\nabla$ and a Higgs field $\Phi \in \Omega^{1,0}(M, \mathrm{End}(E))$. We will often drop the $M$ and write $\Omega^1(\mathfrak u(E)), \Gamma(\mathrm{End}(E)),\dots$  instead of $\Omega^1(M,\mathfrak u(E)),\Gamma(M,\mathrm{End}(E)),\dots$ in order to simplify the notation.  

It will turn out to be useful to think of the pair $(\nabla,\Phi)$ as an element of the cotangent bundle $T^*\mathcal A = \mathcal A \times \Omega^{1,0}(\mathrm{End}(E))$. This identification arises from the correspondence between unitary connections and holomorphic structures on $E$ by assigning to a connection the induced $\bar \partial$-operator and the map $$\Lambda: \Phi \mapsto -2i\int_M \mathrm{tr}(\Phi \wedge -),$$ which induces an isomorphism between the $L^2$-completions of $\Omega^{1,0} (\mathrm{End}(E))$ and  $(\Omega^{0,1} (\mathrm{End}(E))^*$.

The gauge group 
$$\mathcal G = \{ u\in  \Gamma(\mathrm{End}(E))\ | \ u^*u = \mathrm{id}_E \} = \Gamma(\mathrm{U}(E))$$ 
of unitary bundle automorphisms of $E$ acts on $T^*\mathcal A$ (on the right) via 
$$u.(\nabla,\Phi) = (\nabla + u^{-1}\mathrm d^\nabla u, u^{-1}\Phi u),$$
preserving the equations and we are interested in the moduli space of solutions to the equations modulo gauge transformations. We shall explain shortly why we let gauge transformations act on the right.

The Lie algebra of the gauge group is given by skew-adjoint bundle endomorphisms of $E$, i.e. $\mathrm{Lie}(\mathcal G) = \Gamma(\mathfrak u(E)) = \{ \xi\in  \Gamma(\mathrm{End}(E))\ | \ \xi^* = -\xi \}$, and the fundamental vector fields of the action on $T^*\mathcal A$ are given by 
$$X^\xi_{(\nabla,\Phi)} = (\mathrm d^\nabla \xi, [\Phi, \xi]) =: \mathcal D_1(\xi).$$
That is, the linearised action is implemented by the first order differential operator 
$$\mathcal D_1: \mathrm{Lie}(\mathcal G) = \Gamma(\mathfrak u(E)) \to \Omega^1(\mathfrak u(E)) \oplus \Omega^{1,0}(\mathrm{End}(E)).$$
The action of $\mathcal G$ is not free on $T^*\mathcal A$: Any gauge transformation $u$ of the form $e^{i\theta}\mathrm{id}_E$, where $\theta \in \mathbb R$, i.e. $u\in Z(\mathrm U(n))$, the centre of $\mathrm U(n)$  (note that this is \emph{not} the centre of the gauge group $\mathcal G$), lies in the stabiliser of any pair $(\nabla,\Phi)$. So if we allow arbitrary solutions and arbitrary gauge transformations, we cannot expect to produce a well-behaved moduli space. Instead, we restrict attention to solutions with minimal stabiliser and divide the gauge group by the centre of $\mathrm U(n)$. 

We thus say that a pair $(\nabla,\Phi)$ is \emph{irreducible}, if its stabiliser is equal to $Z(\mathrm U(n))$ and write $T^*\mathcal A^*$ for the space of irreducible pairs. We define the \emph{reduced gauge group}, which we denote by $\mathcal G^*$ to be 
$$\mathcal G^* = \mathcal G/Z(\mathrm U(n)).$$
Then $\mathcal G^*$ acts freely on $T^*\mathcal A^*$ and we want to study the moduli space 
$$\mathcal M = \{(\nabla, \Phi)\in T^*\mathcal A^* \ |\  F^\nabla - [\Phi\wedge \Phi^*] = 0= \bar\partial^\nabla \Phi\ \}/\mathcal G^*.$$

In order to apply analytical tools, we work with the Banach space completions of $\mathcal A$, $T^*\mathcal A$ and $\mathcal G^*$ with respect to the $L^2_k$-Sobolev norm  which we denote by $\mathcal A_k$, $T^*\mathcal A_k$ and $\mathcal G_k^*$, where on $T^*\mathcal A = \mathcal A \times \Omega^{1,0}(\mathrm{End}(E))$ we take the direct sum of the respective Sobolev norms on each factor. We write $T^*\mathcal A_k^*$ for the dense open subset of irreducible pairs in $T^*\mathcal A_k$. The Sobolev embedding and multiplication theorems imply  that for $k\geq1$ the gauge group $\mathcal G_{k+1}$ is a smooth Hilbert Lie group and acts smoothly on $\mathcal T^*\mathcal A_k$. Note that the action of the gauge group involves derivatives of the gauge transformations, hence the different Sobolev indices.

The space $\mathcal B_k^*=T^*\mathcal A_k^*/\mathcal G_{k+1}^*$ of gauge equivalence classes of irreducible pairs $(\nabla,\Phi)$ is a smooth infinite-dimensional manifold.  In fact $T^*\mathcal A_k^* \to \mathcal B_k^*$ is principal $\mathcal G_{k+1}^*$-bundle due to the existence of a local \emph{Coulomb gauge}, which provides local slices for the action of the gauge group: 

\begin{proposition}\label{hslice}
Let $(\nabla, \Phi) \in T^*\mathcal A_k^*$ be irreducible. Then there exists a constant $\epsilon(\nabla, \Phi) > 0$, such that if  $(\nabla + A, \Phi + \Psi) \in T^*\mathcal A_{L^2_1}$ with $||A||_{L^4}^2 + ||\Psi||_{L^4}^2 < \epsilon$, there exists a unique gauge transformation $u \in \mathcal G_{k+1}^*$ such that 
$$\mathcal D_1^*(u.(A,\Psi)) = 0.$$
\end{proposition}
This observation is of course not new, it is also used by Hitchin in \cite{Hitchin:1987}.
The proof of this proposition is a standard application of the implicit function theorem using the ellipticity of $\mathcal D_1^*\mathcal D_1$ and works along the lines of proposition 2.3.4 in \cite{DonaldsonKronheimer:1990}. The interpretation of $T^*\mathcal A_k^*$ as a principal $\mathcal G_{k+1}^*$-bundle is the reason why we prefer to have $\mathcal G_{k+1}^*$ acting on the right. 

Recall that the gauge group acts on $\mathfrak u(E)$-valued two-forms by conjugation, i.e. by the adjoint action $\mathrm{Ad}$. Considering the principal $\mathcal G_{k+1}^*$-bundle $T^*\mathcal A_k^* \to \mathcal B_k^*$, we can form the associated vector bundle
$$\mathcal V = T^*\mathcal A_k^* \times_{\mathrm{Ad \mathcal G_{k+1}^*}} (\Omega^2(\mathfrak u(E)) \oplus \Omega^2(\mathrm{End}(E))).$$
Now we interpret the moduli space $\mathcal M$ as the zero locus of a section $G$ of $\mathcal V$, which is defined as follows:
\begin{eqnarray*}
G: T^*\mathcal A_k^* &\to& \Omega^2(\mathfrak u(E)) \oplus \Omega^2(\mathfrak u(E) \otimes \mathbb C)\\
G(\nabla, \Phi) &=& (F^\nabla - [\Phi\wedge \Phi^*], \bar\partial^\nabla \Phi).
\end{eqnarray*}
Note that $G(u.\nabla, u.\Phi) = u^{-1}(G(\nabla, \Phi))u$, i.e. $G$ is equivariant with respect to the actions of $\mathcal G^*$ on $T^*\mathcal A_k^*$ and $\Omega^2(\mathfrak u(E)) \oplus \Omega^2(\mathfrak u(E) \otimes \mathbb C)$, thus it descends to define a section of $\mathcal V$ over $\mathcal B_k^*$. We compute the derivative of $G$ at a point $(\nabla, \Phi)$ to be 
$$\mathrm dG_{(\nabla, \Phi)}(A,\psi) = (\mathrm d^\nabla A - [\Phi \wedge \psi^*] - [\psi\wedge \Phi^*], \bar\partial^\nabla \psi + [A^{0,1}\wedge \Phi]).$$
If $(\nabla,\Phi)$ solves the harmonic map equations, then the tangent space to the moduli space at this point is identified with the first cohomology space of the following deformation complex: 
$$L^2_k(\mathfrak u(E)) \stackrel{\mathcal D_1}{\rightarrow} \Omega^1_{k-1}(\mathfrak u(E)) \oplus \Omega^{1,0}_{k-1}(\mathrm{End}(E)) \stackrel{\mathrm dG_{(\nabla,\Phi)}}{\rightarrow} \Omega^2_{k-2}(\mathfrak u(E))\oplus\Omega^2_{k-2}(\mathrm{End}(E)),$$
where the first map is given by the linearisation of the action, i.e. by the operator $\mathcal D_1$ introduced earlier, and the second map is the derivative of the map $G$. Thus, the tangent space to the moduli space is identified with the orthogonal complement of the image of $\mathcal D_1$ inside the kernel of $\mathrm dG_{(\nabla,\Phi)}$, i.e. with $\ker(\mathcal D_1^* \oplus \mathrm dG_{(\nabla,\Phi)})$. 

Note that the deformation complex is elliptic. Therefore, its cohomology groups are finite-dimensional and its Euler characteristic is given by the index of the operator $\mathcal D_1^* \oplus \mathrm dG_{(\nabla,\Phi)}$. In fact, by irreducibility we always have $\dim H^0 = 1$ for any irreducible $(\nabla,\Phi)$. The problem is, that we are not able to show that the cohomology group $H^2$ of this complex is of constant dimension independent of $(\nabla,\Phi)$. Therefore, the whole moduli space is not expected to be smooth. 

However, for solutions with zero Higgs field, we are able to prove a vanishing theorem and construct therefore a smooth open neighbourhood of the moduli space of flat unitary connections inside the harmonic map moduli space. The precise statement is given in the following proposition.

\begin{proposition}
Let $\nabla \in \mathcal A_k^*$ be an irreducible flat connection, then the image of the differential of $G$ at $(\nabla, 0) \in \mathcal B_k^*$ is of constant codimension independent of the solution $(\nabla,0)$. Moreover, its kernel has dimension $(\dim \mathrm U(n))4(g-1)+4 = 4(n^2(g-1)+1)$. 
\end{proposition}
\begin{proof}
We work on a slice neighbourhood of $(\nabla,0)$ provided by proposition \ref{hslice}, i.e. we restrict $G$ to an appropriate $\epsilon$-ball in $\ker\mathcal D_1^*$. Since we have zero Higgs field, the operator $\mathcal D_1$ is given by 
$$\mathcal D_1(\xi) = (\mathrm d^\nabla \xi, 0) \qquad \xi \in \Gamma(\mathfrak u(E)).$$
The derivative $\mathrm dG$ simplifies to 
$$\mathrm dG_{(\nabla, 0)}(A,\psi) = (\mathrm d^\nabla A, \bar\partial^\nabla \psi).$$
Let us denote this operator by $\mathcal D_2(A,\psi)$.
Furthermore, we have 
$$\mathcal D_2^*(\alpha, \beta) = ((\mathrm d^\nabla)^* \alpha, (\bar\partial^\nabla)^* \beta),$$ 
where $(\alpha,\beta) \in \Omega_{k-2}^2(\mathfrak u(E)) \oplus \Omega_{k-2}^2(\mathrm{End}(E))$. In the following we drop the $k$ to make the notation more readable but of course still work with the Sobolev completions of the relevant spaces of sections. 

In order to proceed, we make use of some elliptic theory.
Since the domain of $G$ is contained in $\ker \mathcal D_1^*$, we have $\mathcal D_2 = \mathcal D_2 + \mathcal D_1^*$. So we have to check that the dimension of the kernel of the adjoint operator $\mathcal D_2^* + \mathcal D_1$ is independent of $\nabla$.

We use the Hodge star to identify $\Omega^2(\mathfrak u(E)) \cong \Omega^0(\mathfrak u(E))$ and analogously for the complex forms. Under this identification, the operator $(\mathrm d^\nabla)^*$ corresponds to $\mathrm d^\nabla$ and the $(\bar\partial^\nabla)^*$ corresponds to $\partial^\nabla$. Moreover we identify $\Omega^1$ with $\Omega^{0,1}$ in the usual way and thus think of the operator 
$$\mathcal D_2^* \oplus \mathcal D_1: \Omega^2(\mathfrak u(E)) \oplus \Omega^2(\mathrm{End}(E)) \oplus \Omega^0(\mathfrak u(E)) \to \Omega^1(\mathfrak u(E)) \oplus \Omega^{1,0}(\mathrm{End}(E)),$$ 
after putting $\Omega^0(\mathfrak u(E)) \oplus \Omega^0(\mathfrak u(E)) \cong \Omega^0(\mathrm{End}(E))$, as the operator 
$$\mathcal D_2^* \oplus \mathcal D_1: \Omega^0(\mathrm{End}(E)) \oplus \Omega^0(\mathrm{End}(E)) \to \Omega^1(\mathfrak u(E)) \oplus \Omega^{1,0}(\mathrm{End}(E)),$$
given by 
$$\left(\begin{array}{cc} 
\mathrm d^\nabla & 0\\0& \partial^\nabla
\end{array}\right).$$
That is, an element $(\xi, \eta) \in\Omega^0(\mathfrak u(E)) \oplus \Omega^0(\mathrm{End}(E))$ lies in the kernel if and only if 
$$\mathrm d^\nabla \xi = 0 \qquad \text{and \ } \qquad \partial^\nabla \eta = 0.$$
Now by irreducibility of $\nabla$, we immediately conclude that $\xi = 0$ in $\mathrm{Lie}(\mathcal G_{k+1}^*)$, i.e. $\xi \in \mathfrak u(1) \mathrm{id}_E$. 
Moreover, $\eta$ is also parallel as can be seen by the following integration by parts argument:
\begin{eqnarray*}
0&=& || \partial^\nabla \eta||_{L^2}^2 \\
&=& -2i\int_M\mathrm{tr}(\partial^\nabla \eta \wedge (\partial^\nabla \eta)^*) \\
 &=&-2i\int_M\mathrm{tr}(\partial^\nabla \eta \wedge -\bar\partial^\nabla (\eta^*))\\
 &=&-2i\int_M\bar\partial \mathrm{tr}((\partial^\nabla\eta)\eta^*) - \mathrm{tr}((\bar\partial^\nabla\partial^\nabla\eta)\eta^*)\\
 &=& -2i\int_M \mathrm{tr}(-(\partial^\nabla\bar\partial^\nabla\eta)\eta^*) \qquad \text{(as $\nabla$ is flat)}\\
 &=& -2i\int_M \mathrm{tr}(\bar\partial^\nabla \eta \wedge (\bar\partial^\nabla \eta)^*)\\
 &=& ||\bar \partial^\nabla \eta||_{L^2}^2.
 \end{eqnarray*}
Thus, since $\nabla$ is assumed to be irreducible, it follows that $\xi \in i\mathbb R \mathrm{id}_E$ and $\eta\in \mathbb C \mathrm{id}_E$. In other words, $H^2$ is of real dimension $3$. 

The statement about the dimension of the kernel of $\mathrm dG$ is obtained from the Atiyah-Singer-Index theorem. The details can be found in \cite{Hitchin:1987}, section 5. This makes sense, since in the case $\Phi = 0$ both complexes here and in \cite{Hitchin:1987} reduce to the same elliptic complex. Note however, that from the above discussion we have $H^0 =1$ and $H^2 = 3$ in the deformation complex. Thus the proof of the proposition is complete.
\end{proof}

\begin{corollary}
Let $\nabla$ be an irreducible flat connection, then on a sufficiently small neighbourhood of $\nabla$ in $\mathcal A$  there exists a $4(n^2(g-1)+1)$ dimensional family of solutions to the harmonic map equations.
\end{corollary}

\section{The Hypersymplectic Geometry of the Moduli Space}
\subsection{The Moment Map Interpretation of the Equations}
We now apply the hypersymplectic quotient construction in an infinite-dimensional setting in order to study the geometry of the moduli space of solutions to Hitchin's harmonic map equations.

On $T^*\mathcal A$ we have a natural complex structure $I$ induced by the Hodge-star operator acting on one-forms on $M$. Together with the $L^2$ inner product we get an indefinite K\"ahler structure on $T^*\mathcal A$. Under the identification $\Omega^{1,0}(\mathrm{End}(E))\cong (\Omega^{0,1}(\mathrm{End}(E)))^*$ discussed earlier, the complex structure induced by the Hodge-star operator is just multiplication by $i$, that is, 
$$I(A,\Phi) = (iA,i\Phi).$$
The indefinite metric reads
$$g((A,\Phi), (B, \Psi)) = \mathrm{Re}\left(2i\int_M \mathrm{tr}(A^*\wedge B - \Phi \wedge \Psi^*)\right).$$
Like on any complex cotangent bundle, we also have the canonical holomorphic symplectic form $\omega_I^\mathbb C$ given by 
$$\omega_I^\mathbb C((A,\Phi), (B,\Psi)) = \Lambda(\Psi)(A) - \Lambda(\Phi)(B) = 2i\int_M \mathrm{tr}( \Phi\wedge B -\Psi \wedge A).$$
This clearly has type $(2,0)$ with respect to the complex structure $I$. We now define endomorphisms $S$ and $T$ of $T(T^*\mathcal A)$ by taking the real and imaginary parts of $\omega_I^\mathbb C$. That is, we  write 
$$\omega_I^\mathbb C = g(S-,-) + ig(T-,-).$$
Noting  that $\mathrm{Re}(\mathrm{tr}(\Phi\wedge A)) = \mathrm{Re}(\mathrm{tr}(A^*\wedge \Phi^*))$, a direct calculation shows that $S$ and $T$ are given by
$$S(A,\Phi) = (\Phi^*,A^*) \qquad T(A,\Phi) = (i\Phi^*,iA^*) = IS(A,\Phi).$$
The action of the gauge group preserves the above flat hypersymplectic structure on $T^*\mathcal A$, and we show now that it admits a hypersymplectic moment map. We only calculate the moment map $\mu_I$ explicitly and then give the other two moment maps without proof in order to avoid repetition as the calculations are very similar.
Recall that the fundamental vector fields of the action are given by 
$$X^\xi_{(\nabla,\Phi)} = \mathcal D_1(\xi) = (\mathrm d^\nabla \xi, [\Phi,\xi]).$$
Now using bi-invariance of the trace inner product and Stokes' theorem, we compute analogously to the Atiyah-Bott calculation (see \cite{AtiyahBott:1982}, and also \cite{Hitchin:1987})
\begin{eqnarray*}
\omega_I(X_\xi, (A,\Psi)) &=& \int_M \mathrm{tr}(\mathrm d^\nabla \xi\wedge A) -\int_M\mathrm{tr}([\Phi-\Phi^*,\xi]\wedge(\Psi-\Psi^*)\\
&=& \int_M\mathrm{tr}((\mathrm d^\nabla A - [\Psi\wedge \Phi^*] - [\Phi \wedge \Psi^*])\xi)\\
&=& \frac{\mathrm d}{\mathrm dt}|_{t=0} \int_M\mathrm{tr}((R^{\nabla + t A} - [\Phi +t\Psi\wedge \Phi^*+t\Psi^*])\xi).
\end{eqnarray*}
Thus, under the identification $\mathrm{Lie}(\mathcal G)^* \cong \Omega^2(\mathfrak u(E))$ via the pairing 
$(\xi, \alpha)\in \mathrm{Lie}(\mathcal G) \times \Omega^2(\mathfrak u(E)) \mapsto \int_M\mathrm{tr}(\alpha\xi)$, we see that the moment map is given by
$$\mu_I(\nabla,\Phi) = \mathrm F^\nabla - [\Phi\wedge\Phi^*].$$
A similar calculation yields that 
$$(\mu_S +i\mu_T)(\nabla,\Phi) = \bar\partial^\nabla \Phi,$$ 
and so the vanishing of the moment map is indeed given by the harmonic map equations, 
\begin{eqnarray*}
\mu_I^\mathbb C(\nabla,\Phi) &=& \bar\partial^\nabla\Phi = 0,\\
\mu_I(\nabla,\Phi) &=& F^\nabla - [\Phi\wedge\Phi^*] =0. 
\end{eqnarray*}
Thus, formally, the moduli space $\mathcal M$ can be interpreted as a hypersymplectic quotient in an infinite-dimensional setting, and so we move now on to study the hypersymplectic geometry of the smooth open set constructed in the previous section.

\subsection{Degeneracies}
Recall from the introduction that if we are given a Lie group $G$ acting on a hypersymplectic manifold $(M,g,I,S,T)$ with hypersymplectic moment map $\mu: M\to \mathfrak g^*\otimes \mathbb R^3$, then typically the quotient $\mu^{-1}(0)/G$ exists as a smooth manifold. However, the hypersymplectic structure is often only defined on the complement of a degeneracy locus. If $\pi\mu^{-1}(0) \to \mu^{-1}(0)/G$ is the quotient map, then the degeneracy locus is given by those points $\pi(p)$ with $p\in \mu^{-1}(0)$ such that $T_p\mathcal O_p\cap (T_p\mathcal O_p)^\perp\neq \{0\}$, where $\mathcal O_p = G.p$ denotes the $G$-orbit through the point $p$. In the following, we will formally describe the degeneracy locus of the hypersymplectic structure on the moduli space $\mathcal M$.

\begin{proposition}
An element $(\nabla,\Phi) \in T^*\mathcal A$ lies in the degeneracy locus if and only if  the kernel of $\mathcal D_1^\dagger\mathcal D_1$ is non-zero. Here $\mathcal D_1: \Gamma(\mathfrak u(E)) \to \Omega^1(\mathfrak u(E)) \oplus \Omega^1(\mathfrak u(E))$ is the operator defined by the infinitesimal gauge action and $\dagger$ denotes the adjoint taken with respect to the split signature metric $g$ on  $ \Omega^1(\mathfrak u(E)) \oplus \Omega^1(\mathfrak u(E))$. This is an elliptic  operator, given by 
$$\mathcal D_1^\dagger\mathcal D_1(\xi) = (\mathrm d^\nabla)^*\mathrm d^\nabla \xi +*[\phi\wedge*[\phi,\xi]],$$
where $\phi = \Phi-\Phi^*$.
\end{proposition}
\begin{proof}
We want to compute the intersection of the tangent space to a gauge orbit with its orthogonal complement. We work with real co-ordinates, $\phi= \Phi - \Phi^*$. The fundamental vector fields of the gauge action are given by 
$$X_{(\nabla,\phi)}^\xi = (\mathrm d^\nabla \xi, [\phi,\xi]) =(\mathrm d^\nabla \xi, \mathrm{ad}(\phi)(\xi))  = \mathcal D_1\xi.$$
We compute the adjoint of $\mathcal D_1$ with respect to the neutral inner product $g$ defined above. Let $(A,\psi) \in \Omega^1(\mathfrak u(E)) \oplus \Omega^1(\mathfrak u(E))$. Then the adjoint is characterised by the property
$$g(\mathcal D_1\xi, (A,\psi)) = \langle\xi, \mathcal D_1^\dagger(A,\psi)\rangle_{L^2}.$$
The only thing we actually have to compute is the adjoint of $\mathrm{ad}(\phi)(\xi)$ with respect to the ordinary $L^2$ inner product. Let $A \in \Omega^1(\mathfrak u(E))$.
\begin{eqnarray*}
\langle\mathrm{ad}(\phi)(\xi), A\rangle_{L^2} &=& -\int_M \mathrm{tr}([\phi,\xi]\wedge*A)\\
&=& - \int_M \mathrm{tr}((\phi\xi - \xi\phi)\wedge *A)\\
&=& -\int_M \mathrm{tr}((\xi (-*A \wedge \phi - \phi \wedge *A))\\
&=& - \int_M\mathrm{tr}(\xi *(-*[\phi\wedge*A]))\\
&=& \langle\xi, -*[\phi\wedge*A]\rangle_{L^2}.
\end{eqnarray*}
So the adjoint is given by $ \mathrm{ad}(\phi)^*(A) = -* \mathrm{ad}(\phi)(*A)$. With this we now compute for $\xi, \eta \in \Gamma(\mathfrak u(E))$: 
\begin{eqnarray*}
g(\mathcal D_1\xi, \mathcal D_1\eta) &=& \langle\mathrm d^\nabla \xi, \mathrm d^\nabla \eta\rangle_{L^2} - \langle\mathrm{ad}(\phi)(\xi), \mathrm{ad}(\phi)(\eta)\rangle_{L^2}\\
&=& \langle(\mathrm d^\nabla)^*\mathrm d^\nabla \xi,  \eta\rangle_{L^2} - \langle (\mathrm{ad}(\phi))^*\mathrm{ad}(\phi)(\xi), \eta\rangle_{L^2}\\
&=& \langle(\mathrm d^\nabla)^*\mathrm d^\nabla \xi +*[\phi\wedge*[\phi,\xi]],\eta\rangle_{L^2}.
\end{eqnarray*}
Thus, we conclude that  $\mathcal D_1\xi$ lies in the orthogonal complement of the tangent space to the gauge orbit through $(\nabla, \phi)$, i.e. in the kernel of $\mathcal D_1^\dagger$ if and only if 
$$(\mathrm d^\nabla)^*\mathrm d^\nabla \xi +*[\phi\wedge*[\phi,\xi]] = 0,$$
as asserted.
\end{proof}
Since $\phi$ is skew-adjoint, we get that the operator $*[\phi\wedge*[\phi,-]]$ is self-adjoint with non-positive eigenvalues. Hence, the  self-adjoint elliptic operator $(\mathrm d^\nabla)^*\mathrm d^\nabla \xi +*[\phi\wedge*[\phi,\xi]]$, being the sum of a non-negative and a non-positive operator, is in general not positive and might a priori have a non-trivial kernel.

\subsection{Product Structures}

On any hypersymplectic manifold $(M,g,I,S,T)$ there is a circle of product structures orthogonal to $I$ given by 
$$S_\theta = \cos \theta S - \sin \theta T, \qquad T_\theta = \sin \theta S + \cos \theta T,$$
or in more compact notation
$$S_\theta + iT_\theta = e^{i\theta}(S + iT).$$
These product structures are integrable and thus for each $\theta$, a hypersymplectic manifold is locally a product of integral submanifolds associated with the distributions given by the $\pm1$-eigenspaces of $S_\theta$ (analogously for $T_\theta$).
In our situation, the hypersymplectic manifold in question is the cotangent bundle of the space of unitary connections and we obtain 
$$S_\theta = \cos \theta S - \sin \theta T =\left( \begin{array}{cc} 0 & -\cos \theta -\sin \theta * \\ -\cos \theta + \sin \theta *& 0 \end{array}\right),$$
where $*$ is, as usual, the Hodge star operator acting on one-forms. Since $*$ squares to $-1$ on one-forms, we use the  suggestive short-hand notation $e^{*\theta} = \cos\theta + \sin\theta *$, analogous to Euler's formula for complex numbers of unit length. Then 
$$S_\theta = \left( \begin{array}{cc} 0 & -e^{*\theta} \\ -e^{-*\theta} & 0 \end{array}\right).$$
Given $\theta \in \mathbb R$ and a solution $(\nabla, \Phi)$of the harmonic map equations, we can associate to it a pair of connections $(\nabla_\theta^+,\nabla_\theta^-)$ given by
$$\nabla_\theta^\pm = \nabla \pm e^{*\theta}(\Phi-\Phi^*).$$
As a consequence of the harmonic map equations, these connections are \emph{flat}. It turns out that this map really implements the local product structure $S_\theta$. In fact, an easy manipulation shows that, identifying $T^*\mathcal A \cong \mathcal A \times \Omega^1(\mathfrak u(E))$, the harmonic map equations for a pair $(\nabla, \phi)$ can be written as 
\begin{eqnarray*}
F^{\nabla + \phi} &=&0\\
F^{\nabla - \phi} &=& 0\\
(\mathrm d^\nabla)^*\phi &=&0.
\end{eqnarray*}
Writing $\phi = \Phi-\Phi^*$ we obtain the original form of the equations. In other words, the Higgs field $\Phi$ is the $(1,0)$-part of $\phi$. 

\begin{proposition}
The map
$$P_\theta: T^*\mathcal A \to \mathcal A \times \mathcal A\qquad (\nabla, \phi) \mapsto (\nabla + e^{*\theta}\phi, \nabla - e^{*\theta}\phi)$$
identifies $(T^*\mathcal A, S_\theta)$ with $(\mathcal A \times \mathcal A, \left(\begin{array}{cc} -1 & 0 \\ 0 & 1\end{array}\right))$ as paracomplex manifolds.
\end{proposition}
\begin{proof}
Let us write $\mathbf s$ for the paracomplex structure on $\mathcal A \times \mathcal A$.
We take the derivative of $P$ and show that $\mathbf s \circ \mathrm{d}P = \mathrm dP \circ S_\theta$.
The derivative of $P_\theta$ is given by 
$$\mathrm dP_\theta =  \left(\begin{array}{cc} 1 & e^{*\theta} \\ 1 & -e^{*\theta}\end{array}\right).$$
Now 
$$\mathrm dP_\theta \circ S_\theta =  \left(\begin{array}{cc} 1 & e^{*\theta} \\ 1 & -e^{*\theta}\end{array}\right)\left( \begin{array}{cc} 0 & -e^{*\theta} \\ -e^{-*\theta} & 0 \end{array}\right) = \left( \begin{array}{cc} -1 & -e^{*\theta} \\ 1 & -e^{*\theta} \end{array}\right).$$
On the other hand 
$$\mathbf s\circ  \mathrm dP_\theta =   \left(\begin{array}{cc} -1 & 0 \\ 0 & 1\end{array}\right)\left(\begin{array}{cc} 1 & e^{*\theta} \\ 1 & -e^{*\theta}\end{array}\right) = \left(\begin{array}{cc} -1 & -e^{*\theta} \\ 1 & -e^{*\theta}\end{array}\right).$$
The inverse of $P_\theta$ is given by 
$$P_\theta^{-1}(\nabla_1, \nabla_2) = (\frac{1}{2}(\nabla_1 + \nabla_2), \frac{e^{-*\theta}}{2}(\nabla_1-\nabla_2)).$$
\end{proof}
For any $\theta$ this map associates to a solution of the harmonic map equations a pair of flat connections. Moreover, this map is gauge equivariant, so it descends to a map on the respective moduli spaces.

However, the map induced by $P_\theta$ in general does not have to be injective on the moduli space. It may happen that different solutions to the harmonic map equations give rise to gauge equivalent pairs of flat connections. However, it will turn out that on our open set constructed in the previous section, it actually is injective. We postpone the proof of this statement to the final section, see corollary \ref{product}. 

We note that the singular points in the image of the map $P_\theta$ have to come from the degeneracy locus:
\begin{proposition}
Let $(\nabla,\phi)$ be a solution to the harmonic map equations.
If $\nabla_\theta^+=\nabla + e^{*\theta}\phi$ or $\nabla_\theta^-=\nabla - e^{*\theta}\phi$ are reducible, then the solution $(\nabla,\phi)$ lies in the degeneracy locus. 
\end{proposition}
\begin{proof}
We prove the proposition in the case $\theta=0$, the case of general $\theta$ is treated analogously.
Suppose $\nabla^+$ is reducible. Then there exists a section $\xi \in \Gamma(\mathfrak u(E))$, which is not a constant multiple of the identity, such that 
$$0=\mathrm d^{\nabla^+}\xi =\mathrm d^\nabla\xi + [\phi,\xi].$$
Now consider 
\begin{eqnarray*}
 (\mathrm d^{\nabla^-})^*\mathrm d^{\nabla^+}\xi &=& -*\mathrm d^{\nabla^-}*\mathrm d^{\nabla^+}\xi\\
&=& -*\mathrm d^{\nabla^-}*(\mathrm d^{\nabla}\xi + [\phi,\xi])\\
&=& -*(\mathrm d^{\nabla^-}*\mathrm d^\nabla\xi - [*\phi\wedge\mathrm d^{\nabla^-}\xi] )\\
&=& -*(\mathrm d^\nabla*\mathrm d^{\nabla}\xi - [\phi\wedge*\mathrm d^{\nabla}\xi] -[*\phi\wedge\mathrm d^\nabla\xi]+ [*\phi\wedge[\phi,\xi]])\\
&=& (\mathrm d^\nabla)^*\mathrm d^\nabla \phi +*[\phi\wedge*[\phi, \xi]],
\end{eqnarray*}
where we used that by the harmonic map equation $\mathrm d^{\nabla^-}*\phi=0$ and in the last line we used the relation $[\phi\wedge*\mathrm d^\nabla \xi]+ [*\phi\wedge\mathrm d^\nabla\xi] =0$ and the Jacobi identity.
\end{proof}
If the norm of the Higgs field $\phi$ is sufficiently small, then we can prove a converse to this proposition. 

\begin{proposition}\label{smallnondeg}
Let $(\nabla, \phi)$ be a solution to the harmonic map equations.
There exists a constant $C>0$ such that if $||\phi||_{L^4}<C$ and there exists a non-trivial solution $\xi\in\mathrm{Lie}(\mathcal G_{k+1}^*)$ to the degeneracy equation, then $\nabla_\theta^+$ is reducible.
\end{proposition}
\begin{proof}
Again, we may assume $\theta=0$. Suppose that there exists $\xi\in \mathrm{Lie}(\mathcal G_{k+1}^*)$, $\xi\neq0$ such that $\mathcal D_1^\dagger\mathcal D_1 \xi =0$. We have seen in the  proof of the above proposition that this means that $$(\mathrm d^{\nabla^-})^*\mathrm d^{\nabla^+}\xi=0.$$ 
Recall that we have an $L^2$-orthogonal eigenspace decomposition 
$$L^2(M,\mathrm{End}(E)) = \ker((\mathrm d^{\nabla^+})^*\mathrm d^{\nabla^+}) \oplus \oplus_{\lambda>0}\mathrm{Eig}((\mathrm d^{\nabla^+})^*\mathrm d^{\nabla^+}, \lambda).$$ 
We decompose $\xi = \xi_0+\xi_1$. Then we have the estimate 
$$\lambda_{min}(\nabla^+)||\xi_1||_{L^2}^2 \leq ||\mathrm d^{\nabla^+}\xi_1||_{L^2}^2,$$
where $\lambda_{min}(\nabla^+)$ is the smallest positive eigenvalue of $(\mathrm d^{\nabla^+})^*\mathrm d^{\nabla^+}$.

Now consider the equation $(\mathrm d^{\nabla^-})^*\mathrm d^{\nabla^+}\xi=0$ and take the $L^2$-inner product with $\xi_1$. Then, using $\nabla^-= \nabla^+-2\phi$ and that the kernel of $(\mathrm d^{\nabla^+})^*\mathrm d^{\nabla^+}$ is the same as the kernel of $\mathrm d^{\nabla^+}$, we find
$$0=\langle \mathrm d^{\nabla^+}(\xi_0+\xi_1),\mathrm d^{\nabla^-}\xi_1\rangle_{L^2} = ||\mathrm d^{\nabla^+}\xi_1||_{L^2}^2 -2 \langle \mathrm d^{\nabla^+}\xi_1,[\phi,\xi_1]\rangle_{L^2}.$$
Thus,
\begin{eqnarray*}
||\mathrm d^{\nabla^+}\xi_1||_{L^2}^2 &=& 2|\langle \mathrm d^{\nabla^+}\xi_1,[\phi,\xi_1]\rangle_{L^2}|\\
&\leq& 4 ||\mathrm d^{\nabla^+}\xi_1||_{L^2}||\phi||_{L^4}||\xi_1||_{L^4} \quad \text{(Cauchy--Schwartz and H\"older inequality)}\\
&\leq& 4\kappa ||\mathrm d^{\nabla^+}\xi_1||_{L^2}||\phi||_{L^4}||\xi_1||_{L^2_1} \quad \text{(Sobolev embedding $L^2_1\to L^4$)}\\
&\leq&  4\kappa||\phi||_{L^4}||\xi_1||_{L^2_1}^2\\
&=& 4\kappa||\phi||_{L^4}(||\xi_1||_{L^2}^2 + ||\mathrm d^{\nabla^+}\xi_1||_{L^2}^2)\\
&=& 4\kappa(1+\lambda_{min}(\nabla^+)^{-1})||\phi||_{L^4}||\mathrm d^{\nabla^+}\xi_1||_{L^2}^2,
\end{eqnarray*}
where $\kappa = \kappa(\nabla^+)$ is the constant from the Sobolev embedding $L^2_1\to L^4$. Thus, if $||\phi||_{L^4}< C:= (4\kappa(1+\lambda_{min}(\nabla^+)^{-1}))^{-1}$, we conclude that $\xi_1=0$, so that $\xi=\xi_0$ lies in the kernel of $\mathrm d^{\nabla^+}$. Hence, since we assume that $\xi\neq0$, i.e. $\xi$ is not a constant multiple of the identity, we conclude that $\nabla^+$ must be reducible.
\end{proof}

It turns out that all product structures on the circle orthogonal to $I$ are equivalent.
\begin{proposition}\label{productstructure}
The circle action $(\nabla, \Phi)\mapsto (\nabla, e^{i\alpha}\Phi)$ induces a paraholomorphic diffeomorphism $(T^*\mathcal A, S_\theta) \cong (T^*\mathcal A, S_{\theta + \alpha})$.
\end{proposition}
\begin{proof}
Let $\tau: \Omega^1(\mathrm{End}(E)) \to \Omega^1(\mathrm{End}(E))$ be the transposition map (or more invariantly the anti-linear involution induced by minus the Cartan involution) $\tau(\Phi) = \Phi^*$. The product structure $S_\theta$ may then be written as
$$S_\theta = \begin{pmatrix} 0 & e^{i\theta}\tau\\ e^{i\theta}\tau & 0\end{pmatrix}.$$
For fixed $\alpha$, the derivative of the map $(\nabla, \Phi)\mapsto (\nabla, e^{i\alpha}\Phi)$ is given by 
$$\begin{pmatrix}1 & 0\\ 0 & e^{i\alpha} \end{pmatrix}.$$
The proof is finished by  a direct calculation:
$$\begin{pmatrix}1 & 0\\ 0 & e^{i\alpha} \end{pmatrix}\begin{pmatrix} 0 & e^{i\theta}\tau\\ e^{i\theta}\tau & 0\end{pmatrix} = \begin{pmatrix} 0 & e^{i\theta+\alpha}\tau\\ e^{i\theta+\alpha}\tau & 0\end{pmatrix}\begin{pmatrix}1 & 0\\ 0 & e^{i\alpha} \end{pmatrix}.$$
\end{proof}

\subsection{Complex Structures}

Recall that on any hypersymplectic manifold, we have a two-sheeted hyperboloid of complex structures, which we may parametrise by $\zeta \in \mathbb{CP}^1\setminus \{|\zeta| = 1\}$: 
$$I_\zeta = \frac{1}{1-|\zeta|^2}\left( (1+|\zeta|^2)I + (\zeta+\bar\zeta)S +i(\zeta-\bar\zeta)T\right).$$
Note that in this notation $I = I_0$. With respect to the complex structure $I_0$, we have already seen that the map assigning to a unitary connection and a Higgs field the associated $\bar\partial$-operator and the $(1,0)$-component of the Higgs field is biholomorphic. It identifies $(\mathcal A \times \Omega^1(\mathfrak u(E)), I_0)$ with the holomorphic cotangent bundle $T^*\mathcal A$ of  the space of $\bar \partial$-operators. But how about the other complex structures? 

\begin{definition}
Let $\lambda \in \mathbb C^*$. A \emph{partial $\lambda$-connection} on a hermitian vector bundle $(E,h)$ is a $\mathbb C$-linear map
$$\nabla^\lambda: \Gamma(E) \to \Omega^{1,0}(E),$$
such that 
$$\nabla^\lambda (fs) = \lambda \partial f \otimes s + f\nabla^\lambda s$$
for all $s \in \Gamma(E)$ and $f \in C^\infty(M)$.
\end{definition} 
To our knowledge, the definition of a $\lambda$-connection is due to Deligne and appeared first in Simpson's work on non-abelian Hodge theory, see for example \cite{Simpson:1991}.
 
We denote by $\mathcal A^\lambda$ the set of partial $\lambda$-connections on $E$, which is an affine space modelled on $\Omega^{1,0}(\mathrm{End}(E))$. We also observe that if $\lambda = 0$, we may think of a $0$-connection as just a $\mathrm{End}(E)$-valued $(1,0)$-form. The complex gauge group $\mathcal G^\mathbb C = \Gamma(\mathrm{GL}(E))$ acts on $\mathcal A^\lambda$ in a natural way by conjugation.
\begin{proposition}
Let $\zeta \in \mathbb C$ with $|\zeta| \neq 1$.The map 
\begin{eqnarray*}
F_\zeta: (T^*\mathcal A, I_\zeta) &\to& (\mathcal A \times \mathcal A^{-i\bar \zeta}, i \oplus i)\\
(\bar\partial^\nabla, \Phi) &\mapsto& (\bar\partial^\nabla -i\bar\zeta\Phi^*, -i\bar\zeta\partial^\nabla +\Phi)
\end{eqnarray*}
is a $\mathcal G$-equivariant holomorphic diffeomorphism.
\end{proposition} 
\begin{proof}
Let $\tau$ be the transposition map introduced in the proof of proposition \ref{productstructure}. Then we may write $I_\zeta \in \mathrm{End}(\Omega^{1,0}(\mathrm{End}(E))\oplus \Omega^{1,0}(\mathrm{End}(E)))$ schematically as the two by two matrix

$$I_\zeta = \frac{1}{1-|\zeta|^2}\begin{pmatrix}(1 +|\zeta|^2)i & 2\bar\zeta \tau \\ 2\bar \zeta\tau & (1+|\zeta|^2)i\end{pmatrix}.$$
The derivative of $F_\zeta$ is given by 
$$\mathrm dF_\zeta(A,\Phi) = \begin{pmatrix} 1 & -i\bar\zeta\tau\\ -i\bar\zeta\tau& 1\end{pmatrix}.$$
Now a direct computation, keeping in mind that $\tau$ is conjugate-linear, gives
$$\mathrm dF_\zeta \circ I_\zeta = i\mathrm dF_\zeta.$$
The equivariance is clear.
\end{proof}

Thinking of $0$-connections as Higgs fields, we see that the map $F_\zeta$ is a direct generalisation of the map $F_0$ given in these co-ordinates by the identity.

In analogy to the case of Higgs bundles, it turns out that apart from $\pm I$ all other complex structures are equivalent.

\begin{proposition}
The complex structures $I_\zeta$, where $\zeta \neq 0,\infty$, are all equivalent.
\end{proposition}
\begin{proof}
Let $\lambda, \zeta \in \mathbb C^*$. Then the map 
$$\mathcal A \times \mathcal A^\lambda \to \mathcal A \times \mathcal A^\zeta,$$
given by the identity on the first factor and multiplication by $\zeta\lambda^{-1}$ on the second factor,  gives the desired biholomorphism.
\end{proof}

On the space $T^*\mathcal A_k$ of pairs $(\bar\partial^\nabla, \Phi)$ of holomorphic structures and Higgs fields we have a natural action of the group $\mathcal G_{k+1}^\mathbb C = L^2_{k+1}(\mathrm{GL}(E))$ of complex gauge transformations.
A rearrangement argument shows that if two solutions $(\nabla_1,\Phi_1)$ and $(\nabla_2,\Phi_2)$ to the harmonic map equations with sufficiently small Higgs fields $\Phi_1,\Phi_2$ are gauge equivalent by a complex gauge transformation, then they must be gauge equivalent by a unitary gauge transformation already. In other words, if two solutions to the harmonic map equations with small Higgs fields are contained in the same $\mathcal G_{k+1}^\mathbb C$-orbit, they must in fact lie in the same $\mathcal G_{k+1}$-orbit. 

\begin{proposition}[Local Uniqueness]
Let $(\nabla_i,\Phi_i)$ be two solutions to the harmonic map equations defined on a hermitian vector bundle $E$. Suppose that there exists a complex gauge transformation $u\in L^2_k(\mathrm{GL}(E))$ such that 
$$(\bar\partial^{\nabla_1},\Phi_1) = u.(\bar\partial^{\nabla_2},\Phi_2).$$
Then $(\nabla_1,\Phi_1)$ and $(\nabla_2,\Phi_2)$ are gauge equivalent by a unitary gauge transformation, provided $\Phi = -\Phi_1^t\otimes 1 + 1\otimes \Phi_2$ satisfies
$$||\Phi\wedge \Phi^*||_{L^2}^2 < \lambda_1(\nabla),$$
where $\nabla$ is the induced connection on $\mathrm E^*\otimes E \cong \mathrm{End}(E)$ with $\nabla_1$ acting on $E^*$ and $\nabla_2$ acting on $E$ and $\lambda_1(\nabla)$ denotes the first non-zero eigenvalue of its associated Laplacian acting on sections of $\mathrm{End}(E)$.
\end{proposition}
\begin{proof}
It is straight-forward to check that $(\nabla,\Phi)$ satisfies the harmonic map equations. The map $u$ being a complex gauge transformation transforming $(\nabla_2,\Phi_2)$ into $(\nabla_1,\Phi_1)$, means that $$\bar\partial^\nabla u = 0,$$ 
if we view $u$ as a section of $\mathrm{End}(E)$.

Moreover, as $\Phi_1$ acts on $E^*$ via $\Phi(\alpha)(v) = -\alpha(\Phi(v))$, it follows that $\Phi u = -u\Phi_1 + \Phi_2u = 0$, since $u^{-1}\Phi_2u= \Phi_1$. 

Since the Laplacian $\Delta^\nabla$ associated with $\nabla$ is elliptic, self-adjoint, and positive, it follows from the compactness of $M$ that $L^2(\mathrm{End}(E))$ decomposes into an orthogonal direct sum of its (finite-dimensional) eigenspaces. Decompose $u = u_0 + u_\perp$, with $u_0$ the orthogonal projection onto $\ker(\Delta^\nabla) = \ker (\mathrm d^\nabla)$ and $u_\perp = u-u_0$. Let $\lambda_1$ be the smallest non-zero eigenvalue of $\Delta^\nabla$.
Now we apply a Weitzenb\"ock argument.

\begin{eqnarray*}
||\mathrm d^\nabla u||_{L^2}^2 &=&||\mathrm d^\nabla u_\perp||_{L^2}^2\\ 
&=&||\partial^\nabla u_\perp||_{L^2}^2 \\
&=& \int_M \mathrm{tr}(F^\nabla u_\perp\wedge *u_\perp)\\
&=& \int_M \mathrm{tr}([\Phi\wedge \Phi^*]u_\perp\wedge *u_\perp)\\
&=& \int_M \mathrm{tr}(\Phi\wedge\Phi^*u_\perp\wedge*u_\perp)\\
&\leq& ||\Phi\wedge\Phi^*||_{L^2}^2||u_\perp||_{L^2}^2\\
&<& \lambda_1||u_\perp||_{L^2}^2.
\end{eqnarray*}

On the other hand, we have that 
$$\lambda_1 ||u_\perp||_{L^2}^2 \leq \langle\Delta^\nabla u_\perp,u_\perp\rangle_{L^2} =  ||\mathrm d^\nabla u_\perp||_{L^2}^2 < \lambda_1||u_\perp||_{L^2}^2.$$
So we conclude that $u_\perp = 0$ and hence $u$ is parallel with respect to $\nabla$ and moreover 
$$0 = ||\mathrm d^\nabla u||_{L^2}^2 = \langle\Phi\wedge\Phi^*u,u\rangle =||\Phi^*u||^2,$$
so $\Phi^*u=0$. 
Now we define the unitary gauge transformation 
$$\tilde u = u(u^*u)^{-\frac{1}{2}}.$$
Then $\tilde u$ is also parallel and hence gauges $\nabla_2$ to $\nabla_1$. Furthermore, since $\Phi u = 0 = \Phi^*u$, it follows that $\Phi u^* = 0$ and hence $\Phi \tilde u =0$, i.e
$$\tilde u^{-1}\Phi_2\tilde u = \Phi_1.$$
\end{proof}

\begin{remark}
This proposition shows that  a small neighbourhood of the moduli space of irreducible flat connections in the moduli space of solutions to the harmonic map equations may be identified with an appropriate moduli space of $\lambda-$connections modulo complex gauge transformations.
\end{remark}

\section{The Riemannian Geometry of the Moduli Space of Connections}
In this section we investigate the equations from an alternative point of view, which originates in the following observation (see the appendix of \cite{Hitchin:1990}). 

Let $k\geq1$  and consider the space $\mathcal A_k^*$ of irreducible unitary connections of Sobolev class $k$ on the hermitian vector bundle $E$. As  we have seen, this is an infinite-dimensional affine space modelled on $\Omega^1_k(\mathfrak u(E))$. Equipped with the $L^2$ inner product, we may view this as a flat infinite-dimensional Hilbert manifold. The Hilbert Lie group $\mathcal G^*_{k+1}$ of reduced gauge transformations acts freely on $\mathcal A_k^*$ by isometries. If we put the quotient metric on $\mathcal A_k^*/\mathcal G_{k+1}^*$, then the projection $\mathcal A_k^* \to \mathcal A_k^*/\mathcal G_{k+1}^*$ becomes a Riemannian submersion. It then turns out that the harmonic map equations can be given a natural interpretation in the terms of the Riemannian geometry of $\mathcal A_k^*/\mathcal G_{k+1}^*$. 

Since $\mathcal A_k^*$ is just an affine space, geodesics are given by straight lines. Now since $\mathcal A_k^* \to \mathcal A_k^*/\mathcal G_{k+1}^*$ is a Riemannian submersion, geodesics on the base can, at least locally, be lifted to horizontal geodesics on $\mathcal A_k^*$, i.e. geodesics orthogonal to the $\mathcal G_{k+1}^*$-orbits. At a point $\nabla \in \mathcal A_k^*$, the tangent space to the $\mathcal G_{k+1}^*$-orbit is given by the image of $\mathrm d^\nabla: L^2_{k+1}(\mathfrak u(E)) \to \Omega^1_k(\mathfrak u(E))$. Thus, a geodesic $\gamma(t) = \nabla + t\phi$ is horizontal at $\nabla$ if and only if $(\mathrm d^\nabla)^*\phi = 0$. Notice that since $[\phi\wedge*\phi] = 0$, it then automatically follows that $(\mathrm d^{\nabla+t\phi})^*\phi = 0$, and the geodesic is in fact horizontal for all $t$. Therefore, in order to specify a horizontal geodesic, we need to fix a point  $\nabla$ on the geodesic and a horizontal direction vector $\phi \in \ker(\mathrm d^\nabla)^*$. In other words, we can think of the space $T^*\mathcal A_k^*$ as the space of horizontal geodesic segments on $\mathcal A_k^*$. 

Decomposing $\phi$ into its $(1,0)$ and $(0,1)$ parts, i.e. writing $\phi = \Phi -\Phi^*$ with $\Phi \in \Omega^{1,0}(\mathrm{End}(E))$, the harmonic map equations are equivalent to the system
\begin{eqnarray*}
F^{\nabla + \phi} &=&0\\
F^{\nabla - \phi} &=& 0\\
(\mathrm d^\nabla)^*\phi &=&0.
\end{eqnarray*}
In other words, we may interpret the harmonic map equations as the equation of a geodesic on $\mathcal A_k^*/\mathcal G_{k+1}^*$ whose endpoints $\nabla\pm\phi$ are contained in the moduli space of flat unitary connections, which we denote by $\mathcal N_k$. We summarise this discussion in the following proposition. 

\begin{proposition}[\cite{Hitchin:1990}]
Solutions to the harmonic map equations modulo gauge equivalence are in one-to-one correspondence with geodesics on the moduli space $\mathcal A_k^*/\mathcal G_{k+1}^*$ of irreducible unitary connections whose endpoints are contained in the moduli space $\mathcal N_k$ of flat connections. 
\end{proposition}
\subsection{An Existence Theorem}
In this setting the existence of solutions with sufficiently small Higgs fields follows rather easily. Observe that since $(\mathrm d^{\nabla - \phi})^*\phi = 0$, the connection $\nabla^+= \nabla +\phi$ is in Coulomb gauge with respect to $\nabla^-=\nabla -\phi$ (and vice versa of course). Thus, the existence of solutions with sufficiently small Higgs fields follows from the existence of a local Coulomb gauge.

\begin{proposition}[\cite{DonaldsonKronheimer:1990}]\label{slice}
Let $k\geq1$ and let $\nabla \in \mathcal A^*_k$ be irreducible. Then there exists a constant $\epsilon(\nabla) > 0$, such that if  $\nabla + A \in \mathcal A_k$ with $A\in\Omega^1(\mathfrak u(E))$ satisfies $||A||_{k}^2 < \epsilon$, there exists a unique gauge transformation $u \in \mathcal G_{k+1}^*$ such that 
$$(\mathrm d^\nabla)^*(u^{-1}Au + u^{-1}\mathrm d^\nabla u) = 0.$$
\end{proposition}

Together with the elementary fact that flatness is a gauge-invariant condition, this proposition proves the existence of short geodesics, i.e. solutions to the harmonic map equations with small Higgs fields. Moreover, the product structure $S$ on $T^*\mathcal A_k^*$ has a natural interpretation, as it assigns to a geodesic segment linking $\nabla^-$ to $\nabla^+$ its endpoints. Our aim is now to prove that this map is injective on sufficiently short geodesics, i.e. solutions with sufficiently small Higgs fields. The following theorem asserts, that any connection has a neighbourhood in which any two points can be linked by a unique horizontal geodesic. Thus, in this neighbourhood geodesic segments are uniquely determined by their endpoints.

\begin{theorem}\label{ball}
Let $\nabla\in\mathcal A_k^*$ be an irreducible connection and let $\nabla_i=\nabla +A_i\in\mathcal A_k$ for  $i=1,2$. Then there exists a constant $C>0$ depending only on $\nabla$ such that if $||A_1||_k+||A_2||_k <C$, then there exists a unique gauge transformation $u\in \mathcal G_{k+1}^*$ such that $u.\nabla_2$ is in Coulomb gauge with respect to $\nabla_1$. 
\end{theorem}
\begin{proof}
This is an application of the implicit function theorem. Consider the map 
$$F: \Omega^1_k(\mathfrak u(E)) \times \Omega^1_k(\mathfrak u(E)) \times \mathrm{Lie}(\mathcal G_{k+1}^*) \to \Omega^0_{k-1}(\mathfrak u(E))/i\mathbb R\mathrm{id}_E$$
given by
\begin{eqnarray*}
F(A_1,A_2,\xi) &=& (\mathrm d^{\nabla+A_1})^*(\exp(-\xi).\nabla_2-\nabla_1)\\
&=& (\mathrm d^{\nabla+A_1})^*(A_2-A_1+\exp(-\xi)\mathrm d^{\nabla+A_2}\exp(\xi)).
\end{eqnarray*}
Thus, we want to solve the equation $F(A_1,A_2,\xi) = 0$ as a function in $\xi$. The partial derivative of $F$ with respect to the $\mathrm{Lie}(\mathcal G_{k+1}^*)$-component at the point $(A_1,A_2,\xi) = (0,0,0)$ is given by 
$$\mathrm D_3F|_{(0,0,0)}: \mathrm{Lie}(\mathcal G_{k+1}^*) \to \Omega_{k-1}^0(\mathfrak u(E))/i\mathbb R\mathrm{id}_E, \quad\mathrm D_3F|_{(0,0,0)}\eta = (\mathrm d^\nabla)^*\mathrm d^\nabla\eta,$$
which is an isomorphism since we assume $\nabla$ to be irreducible. 

Therefore, the implicit function theorem guarantees the existence of a solution $\xi$ provided $||A_1||_k+||A_2||_k <C$, where $C>0$ is a positive constant depending only on $\nabla$.

To prove uniqueness, assume that we have $\nabla_2 = \nabla + A_2$ and $\tilde\nabla_2 = \nabla+B_2$ in Coulomb gauge with respect to $\nabla_1=\nabla+A_1$ with $\tilde\nabla_2 = u.\nabla_2$ for some gauge transformation $u\in\mathcal G_{k+1}^*$. This means
$$\mathrm d^\nabla u =uB_2- A_2u.$$
Now we apply $(\mathrm d^{\nabla_1})^*=-*\mathrm d^{\nabla_1}*$ to this equation and use the Coulomb gauge condition $(\mathrm d^{\nabla_1})^*A_2 = 0 = (\mathrm d^{\nabla_1})^*B_2$ to obtain
$$(\mathrm d^{\nabla_1})^*\mathrm d^\nabla u = -*\mathrm d^{\nabla_1}*(uB_2-A_2u) = -*(\mathrm d^{\nabla_1}u\wedge*B_2 +*A_2\wedge\mathrm d^{\nabla_1}u).$$
Once again, we make use of the eigenspace decomposition 
$$L^2(M,\mathrm{End}(E)) = \ker((\mathrm d^\nabla)^*\mathrm d^\nabla) \oplus \oplus_{\lambda>0}\mathrm{Eig}((\mathrm d^\nabla)^*\mathrm d^\nabla, \lambda).$$
We decompose $u= u_0 + u_1$ according to this decomposition and find
$$||\mathrm d^\nabla u||_{L^2}^2  \geq \lambda_{min}(\nabla)||u_1||_{L^2}^2.$$
That is,
$$||u_1||_{L^2}^2 \leq \lambda_{min}^{-1}(\nabla)||\mathrm d^\nabla u||_{L^2}^2.$$
We take the $L^2$-inner product of $(\mathrm d^{\nabla_1})^*\mathrm d^\nabla u$ with $u_1$ and calculate using the above relation as well as  the Cauchy-Schwarz and H\"older inequalities
\begin{eqnarray*}
\langle (\mathrm d^{\nabla_1})^*\mathrm d^\nabla u, u_1\rangle_{L^2} &=&\langle -*(\mathrm d^{\nabla_1}u\wedge*B_2 +*A_2\wedge\mathrm d^{\nabla_1}u), u_1\rangle_{L^2}\\
&=& \int_M \mathrm{tr}((-*(\mathrm d^{\nabla_1}u\wedge*B_2 +*A_2\wedge\mathrm d^{\nabla_1}u))\wedge *u_1^*)\\
&=& -\int_M \mathrm{tr}(\mathrm d^{\nabla_1}u\wedge *B_2u_1^* + *u_1^*A_2\wedge \mathrm d^{\nabla_1}u)\\
&=& \langle \mathrm d^{\nabla_1}u, u_1B_2^*-A_2^*u_1\rangle_{L^2}\\
&\leq& ||\mathrm d^{\nabla_1} u||_{L^2} ||u_1||_{L^4}(||A_2||_{L^4} +||B_2||_{L^4})\\
&=& ||\mathrm d^{\nabla_1} u_1||_{L^2} ||u_1||_{L^4}(||A_2||_{L^4} +||B_2||_{L^4}).
\end{eqnarray*}
In the last line we used that $\nabla$ is irreducible, which implies $u_0 \in \mathbb C\mathrm{id}_E\subset \ker\mathrm d^{\nabla_1}$. Thus, $\mathrm d^{\nabla_1}u= \mathrm d^{\nabla_1}u_1$.
On the other hand
$$\langle (\mathrm d^{\nabla_1})^*\mathrm d^\nabla u, u_1\rangle_{L^2} = ||\mathrm d^\nabla u||_{L^2}^2 + \langle \mathrm d^\nabla u, [A_1,u_1]\rangle_{L^2}.$$
Using the Sobolev embedding $L^2_1\to L^4$ and the Cauchy-Schwartz and H\"older inequalities, as well as the inequalities $||\mathrm d^\nabla u||_{L^2} \leq ||u||_{L^2_1}$ and $||u_1||_{L^2}^2 \leq \lambda_{min}^{-1}||\mathrm d^\nabla u||_{L^2}^2$, we calculate
\begin{eqnarray*}
||\mathrm d^\nabla u_1||_{L^2}^2 &=& ||\mathrm d^\nabla u||_{L^2}^2  \\
&\leq&\langle (\mathrm d^{\nabla_1})^*\mathrm d^\nabla u, u_1\rangle_{L^2} + |\langle \mathrm d^\nabla u, [A_1,u_1]\rangle_{L^2}|\\
&\leq& ||\mathrm d^{\nabla_1} u_1||_{L^2} ||u_1||_{L^4}(||A_2||_{L^4} +||B_2||_{L^4}) +  |\langle \mathrm d^\nabla u, [A_1,u_1]\rangle_{L^2}|\\
&\leq& ||\mathrm d^{\nabla_1} u_1||_{L^2} ||u_1||_{L^4}(||A_2||_{L^4} +||B_2||_{L^4}) + 2||\mathrm d^\nabla u||_{L^2}||u_1||_{L^4}||A_1||_{L^4}\\
&\leq& (||\mathrm d^{\nabla} u_1||_{L^2}+2||u_1||_{L^4}||A_1||_{L^4}) ||u_1||_{L^4}(||A_2||_{L^4} +||B_2||_{L^4}) \\
& & + 2||\mathrm d^\nabla u||_{L^2}||u_1||_{L^4}||A_1||_{L^4}\\
&\leq& \kappa(||u_1||_{L^2_1}+2\kappa||u_1||_{L^2_1}||A_1||_{L^4}) ||u_1||_{L^2_1}(||A_2||_{L^4} +||B_2||_{L^4}) \\
& & + 2\kappa|| u_1||_{L^2_1}||u_1||_{L^2_1}||A_1||_{L^4}\\
&=&||u_1||_{L^2_1}^2\kappa(2||A_1||_{L^4} +(1+2\kappa||A_1||_{L^4})(||A_2||_{L^4} + ||B_2||_{L^4}))\\
&\leq& ||\mathrm d^\nabla u_1||_{L^2}^2(1+\lambda_{min}^{-1})\kappa(2||A_1||_{L^4} +(1+2\kappa||A_1||_{L^4})(||A_2||_{L^4} + ||B_2||_{L^4})),
\end{eqnarray*}
where $\kappa = \kappa(\nabla)$ is the constant from the Sobolev embedding $L^2_1\to L^4$.
Hence, if $(2||A_1||_{L^4} +(1+2\kappa||A_1||_{L^4})(||A_2||_{L^4} + ||B_2||_{L^4})) < ((1+\lambda_{min}^{-1})\kappa)^{-1}$ we conclude that $\mathrm d^\nabla u_1=0$, i.e. $u=1 \in \mathcal G_{k+1}^*$.
\end{proof}
Observe that the uniqueness-part of the proof only requires $L^4$-smallness. 

\begin{corollary}\label{product}
Every irreducible flat $L^2_k$-connection has a small neighbourhood on which a horizontal geodesic with flat endpoints is uniquely determined by the gauge equivalence classes of its endpoints. In particular, an open subset of the moduli space of solutions to the harmonic map equations can be identified with a neighbourhood of the diagonal in the product of the moduli space of flat connections with itself. 
\end{corollary}
\begin{proof}
Suppose $\nabla^0\in\mathcal A_k^*$ is flat and let $(\nabla,\Phi)$ be a solution to the harmonic map equations with associated pair of flat connections $\nabla^\pm = \nabla \pm \phi$. Write $\nabla = \nabla^0+A$. Let $(\tilde\nabla,\tilde\Phi)$ be another solution to the harmonic map equations such that $\tilde\nabla^\pm = \tilde\nabla\pm\tilde\phi$ is gauge equivalent to $\nabla^\pm$. That is, there exist gauge transformations $u_+$ and $u_-$ such that 
$$u_\pm.\nabla^\pm = \tilde\nabla^\pm.$$
Since $\nabla^+$ (respectively $\tilde\nabla^+$) is in Coulomb gauge with respect to $\nabla^-$ (respectively $\tilde\nabla^-$), this means that $(u_-^{-1}).\tilde\nabla^+$ is in Coulomb gauge with respect to $u_-^{-1}.\tilde\nabla^- = \nabla^-$. Thus, $\nabla^-$ is in Coulomb gauge with respect to both $(u_-^{-1}).u_+.\nabla^+$ and $\nabla^+$. 

In analogy to the notation in the proof of theorem \ref{ball}, we write $\nabla^-=\nabla^0+A_1$, $\nabla^+ = \nabla^0+A_2$ and $u_-^{-1}.u_+.\nabla^+ = \nabla^0 + B_2$. Then we have 
\begin{eqnarray*}
A_1&=& A-\phi,\\
A_2 &=& A+\phi,\\
B_2 &=& u_-(u_+^{-1}(A+\phi)u_+ + u_+^{-1}\mathrm d^{\nabla^0}u_+)u_-^{-1} + u_-\mathrm d^{\nabla^0}u_-^{-1}.
\end{eqnarray*}
Since the $L^4$-norm is $\mathrm{Ad}$-invariant, we therefore have
$$||B_2||_{L^4} \leq ||A+\phi||_{L^4} + ||u_+^{-1}\mathrm d^{\nabla_0}u_+||_{L^4} + ||u_-^{-1}\mathrm d^{\nabla^0}u_-||_{L^4}.$$
On the other hand, the condition $u_\pm.\nabla^\pm = \tilde\nabla^\pm.$ implies that 
$$u_\pm^{-1}\mathrm d^{\nabla^0}u_\pm = \tilde A\pm\tilde\phi - u_\pm^{-1}(A\pm\phi)u_\pm.$$
Thus, 
$$||B_2||_{L^4} \leq 2||A+\phi||_{L^4} + ||\tilde A+\tilde \phi||_{L^4} + ||\tilde A-\tilde\phi||_{L^4} + ||A-\phi||_{L^4}.$$
Hence, if $||\tilde A||_{L^4}, ||\tilde\phi||_{L^4}, ||A||_{L^4}, ||\phi||_{L^4}$ are sufficiently small, we can ensure that the conditions at the end of the  proof of theorem \ref{ball} are satisfied and therefore conclude
$$u_+=u_-,$$
which implies that $(\nabla,\Phi)$ and $(\tilde\nabla, \tilde\phi)$ are gauge equivalent as solutions to the harmonic map equations.
\end{proof}

We have seen that geodesics on $\mathcal A_k^*/\mathcal G_{k+1}^*$ become unique once they are short enough (with respect to the $L^4$-norm). Thus, the space of horizontal geodesics with flat endpoints locally looks like a product of a small ball in the moduli space of flat connections with itself. This subset is open in the harmonic map moduli space, since we know that near an irreducible flat connection the moduli space has dimension $4(n^2(g-1)+1)$, which equals twice the dimension of the moduli space of irreducible flat unitary connections. 

Our argument shows that a subset in the moduli space of solutions to the harmonic map equations, on which the hypersymplectic structure is non-degenerate, can be identified with a neighbourhood of the diagonal in the product of the moduli space of flat irreducible connections with itself. This is in general an open set in the moduli space of solutions to the harmonic map equations. Actually, we have shown that any two irreducible flat connections which lie in a sufficiently small open ball inside $\mathcal A_k^*$ determine up to simultaneous gauge transformations a unique horizontal geodesic linking their gauge equivalence classes.

\subsection{Conjugate Points and the Degeneracy Locus}
Recall that on a Riemannian manifold a Jacobi field is a tangent vector to the space of geodesics, i.e. given a geodesic $\gamma$ , a Jacobi field is a vector field along $\gamma$ of the form 
$$ J(t) = \frac{\mathrm d}{\mathrm ds}|_{s=0} \gamma_s(t),$$
where $\gamma_s$ is a family of geodesics such that $\gamma = \gamma_0$. Two points on a geodesic are \emph{conjugate} if there exists a non-trivial Jacobi field along $\gamma$ which vanishes at these two points. 
\begin{theorem}
Let $\gamma(t)$, $t\in[0,1]$ be a horizontal geodesic segment with flat endpoints, such that the connection $\gamma(1/2)$ is irreducible. If  the endpoints of $\gamma$ are conjugate,  then the hypersymplectic structure on the moduli space of harmonic maps is degenerate at $\gamma$. 
\end{theorem}
The endpoints of $\gamma$ being conjugate means that  there exists  a one-parameter family  $\gamma(s,t)$ of horizontal geodesics with flat endpoints, such that  $\gamma(t) = \gamma(0,t)$ and the endpoints are gauge-equivalent for all $s$. Equivalently, there exists a Jacobi field along $\gamma$ which is tangent to the gauge orbits through the endpoints of $\gamma$ .
\begin{proof}
Let $\nabla^- = \gamma(0,0)$ and $\nabla^+ = \gamma(0,1)$.
Consider the Jacobi field $$Y(t) = \frac{\partial}{\partial s}|_{s=0}\gamma(s,t).$$ Then by assumption, $Y(0)$ and $Y(1)$ are tangent to the gauge orbit through $\gamma(0)$ and $\gamma(1)$ respectively. This means, there are Lie algebra elements $\xi^\pm \in \Gamma(\mathfrak u(E))$ such that
$$Y(0) = \mathrm d^{\nabla^-}\xi^- \qquad Y(1) = \mathrm d^{\nabla^+}\xi^+.$$
Tracking through our correspondence, we write $\nabla = \frac{1}{2}(\nabla^+ + \nabla^-)$, $\phi = \frac{1}{2}(\nabla^+ -\nabla^-) = \Phi - \Phi^*$. In other words
$$\gamma(s,t) = \nabla(s) +(2t-1)\phi(s),$$
and so $\nabla = \nabla(0)$ and $\phi = \phi(0)$. We write $\nabla(s) = \nabla + A(s)$.
In this notation 
$$Y(0) = \dot A - \dot \phi \qquad Y(1) = \dot A + \dot \phi.$$
This gives
\begin{eqnarray*}
\dot A - \dot \phi &=& \mathrm d^\nabla \xi^- - [\Phi-\Phi^*,\xi^-] \\
\dot A + \dot \phi&=& \mathrm d^\nabla \xi^+ + [\Phi-\Phi^*,\xi^+].
\end{eqnarray*}
It follows that
\begin{eqnarray*}
2\dot A &=& \mathrm d^\nabla(\xi^+ + \xi^-) + [\Phi-\Phi^*, \xi^+-\xi^-]\\
2\dot \phi &=&\mathrm d^\nabla(\xi^+ - \xi^-) + [\Phi-\Phi^*, \xi^++\xi^-]. 
\end{eqnarray*}
This means, that in complex co-ordinates $2(\dot A, \dot \phi)$ is represented by the point
\begin{eqnarray*}
(2\dot A^{0,1}, 2\dot \phi^{1,0}) &=& (\bar\partial^\nabla (\xi^++\xi^- -[\Phi^*, \xi^+-\xi^-], \partial^\nabla(\xi^+-\xi^-) +[\Phi, \xi^++\xi^-])\\
&=&(\bar\partial^\nabla (\xi^++\xi^- , [\Phi, \xi^++\xi^-]) + ( -[\Phi^*, \xi^+-\xi^-], \partial^\nabla(\xi^+-\xi^-))\\
&=& X^{\xi^++\xi^-} + SX^{\xi^+-\xi^-}. 
\end{eqnarray*}
It is then easy to check that $\xi = \xi^+ - \xi^-$ defines an element of the degeneracy space at $(\nabla,\phi)$. Indeed, since $(\dot A,\dot \phi)$ is tangent to the moduli space as well as $X^{\xi^++\xi^-}$, the vector $SX^\xi$ has to be tangent to the solution space, too. This means it solves the linearised harmonic map equations, which in this situation precisely yield the degeneracy equation for $\xi$.

This shows that if the endpoints of a geodesic segment corresponding to a solution to the harmonic map equations are \emph{conjugate}, then the hypersymplectic structure is degenerate at this point of the moduli space. In this way, we see that flat connections contained in the image of the \emph{cut locus} of the Riemannian manifold $\mathcal A_k^*/\mathcal G_{k+1}^*$ under the exponential map will also be contained in the degeneracy locus of the hypersymplectic structure.
\end{proof}

On a neighbourhood of the moduli space of flat connections a converse to this theorem is true.

\begin{proposition}
If $(\nabla,\phi)$ is a solution to the harmonic map equations, then the hypersymplectic structure is non-degenerate at this point in the moduli space $\mathcal M$, provided that $(\nabla,\phi)$ is sufficiently close to a flat irreducible connection.
\end{proposition}
\begin{proof}
Let $(\nabla,\phi)$ be a solution to the harmonic map equations and let $\nabla^0$ be an irreducible flat connection, so that we can write $\nabla = \nabla^0+A$. As before, we denote by $\lambda_{min}(\nabla^0)$ be the smallest non-zero eigenvalue of $(\mathrm d^{\nabla^0})^*\mathrm d^{\nabla^0}$ acting on sections of $\mathfrak u(E)$. By proposition \ref{smallnondeg} we have to show that $\nabla^+=\nabla+\phi$ is irreducible. Assume that $\eta \in L_2^2(\mathfrak u(E))$ lies in the kernel of $\nabla^+$, i.e.
$$\mathrm d^{\nabla^+} \eta = 0.$$
Then an argument similar to the one in the proof of \ref{smallnondeg} shows that $\eta$ has to lie in the kernel of $\nabla^0$ already, provided $||A+\phi||_{L^4}< (\kappa(1+\lambda_{min}(\nabla^0)^{-1}))^{-1}$. This then implies that $\eta$ has to vanish in $\mathrm{Lie}(\mathcal G^*)$ and so $\nabla^+$ is irreducible. By proposition \ref{smallnondeg} this means that the point $(\nabla,\phi)$ does not belong to the degeneracy locus, if $||\phi||_{L^4}$ is sufficiently small. 
\end{proof}

\section{Conclusion and Final Comments}

As we have already remarked, with minor modifications the proofs above still work if we replace $U(n)$ by an arbitrary compact Lie group $G$.
We have shown that there exists a hypersymplectic structure on a suitable neighbourhood of the moduli space of irreducible flat $G$-connections inside the moduli space of solutions to the gauge theoretic harmonic map equations. 

We have seen that with a solution of the equations we can associate in a natural way a pair of flat unitary connections over the compact Riemann surface $M$.
Locally, the two flat connections are \emph{trivial}, and hence they have to be gauge equivalent by a local gauge transformation. That is, in a parallel trivialisation (with respect to $\nabla^+$ say), their difference, which equals $2\phi$, has to be of the form $u^{-1}\mathrm du$, for a smooth map $u: U\subset M \to G$, which as a consequence of the holomorphicity of $\Phi$ is \emph{harmonic}. This construction produces a harmonic section of the flat bundle of groups associated with the action of $G \times G$ on $G$ given by $(a,b).u = a^{-1}ub$ equipped with the product connection $(\nabla^+,\nabla^-)$. In other words, the moduli space parametrises harmonic sections of flat $G\times G$-bundles. For zero Higgs field the two connections coincide and there are no non-trivial harmonic sections due to the irreducibility of the flat connection. Thus, we have the moduli space of irreducible flat connections embedded in a natural way.  

In fact, it follows from theorem \ref{ball} that we may naturally interpret this open set inside the moduli space of solutions to the harmonic map equations as the \emph{paracomplexification} of the moduli space of flat $G$-connections. The local product structure $S$ exhibits this open set as a neighbourhood of the diagonal inside the product of the moduli space of flat $G$-connections with itself.  This is the split signature analogue of the fact that the Higgs bundle moduli space with the complex structure $J$ is the moduli space of flat $g^\mathbb C$-connections. Note that $G\times G$ is the paracomplexification of $G$, and so we see that the analogy is actually very close of the case of Higgs bundles: With respect to the paracomplex structure $S$ the open set inside the moduli space describes flat paracomplex connections. 

\section*{Acknowledgements}
This work is based on parts of the author's DPhil thesis \cite{Roeser:2012}.
Special thanks are due to Prof. Andrew Dancer for suggesting this DPhil project and for his support, guidance and encouragement over the past three years.
The author is grateful to  Prof. Nigel Hitchin and Prof. Lionel Mason for a number of useful remarks. The author would also like to thank the anonymous referee for useful comments that helped to improve the manuscript. Most of this work was carried out at the Mathematical Institute of the University of Oxford supported by a DPhil studentship of the Engineering and Physical Sciences Research Council. The author  wishes to thank the University of M\"unster and the SFB 878 ``Groups, Geometry and Actions" for research support.

\bibliographystyle{elsarticle-num}

\end{document}